\documentclass[12pt,oneside]{article}
\usepackage{amsmath,amssymb,amsfonts,amsthm}

\UseRawInputEncoding
  \write18{shell_command}
\def\blue#1{\textcolor[rgb]{0.0,0.0,1.0}{#1}}

\newcommand{\T}{{\mathcal T}}
\newcommand{\C}{{\mathcal C}}

\newcommand{\Real}{\mathbb R}

\newcommand{\abs}[1]{\left\vert#1\right\vert}

\newcommand{\To}{\longrightarrow}

\newcommand{\tm}{\T M}

\def\pa{\partial}
\def\paa{\dot{\partial}}

\def\+{\!+\!}

\def\={\!=\!}
\def\<{\!<\!}
\def\>{\!>\!}

 \let\oldmarginpar\marginpar
\renewcommand\marginpar[1]{\oldmarginpar[\raggedleft\footnotesize #1]%
  {\blue{\raggedright \footnotesize \fbox{
      \begin{minipage}{1.0\linewidth}
        #1
      \end{minipage}
}}}}

\numberwithin{equation}{section} 
\numberwithin{figure}{section} 

\theoremstyle{plain}
\newtheorem*{theorem*}{Theorem}
\newtheorem{theorem}{Theorem}[section]
\newtheorem{lemma}[theorem]{Lemma}
\newtheorem{proposition}[theorem]{Proposition}
\newtheorem{corollary}[theorem]{Corollary}
\theoremstyle{definition}
\newtheorem{definition}[theorem]{Definition}
\newtheorem{property}[theorem]{Property}
\theoremstyle{remark}
\newtheorem{example}{Example}
\newtheorem{remark}[theorem]{Remark}
\newtheorem*{acknowledgement*}{Acknowledgement}

\begin{document}

\title{ \bf   Finsler surfaces with  vanishing $T$-tensor}

\author{S. G. Elgendi }
\date{}

\maketitle
\vspace{-1.20cm}

\begin{center}
{Department of Mathematics, Faculty of Science,\\
Islamic University of Madinah, Madinah, Saudi Arabia}
\vspace{-8pt}
\end{center}

\begin{center}
{Department of Mathematics, Faculty of Science,\\
Benha University, Benha, Egypt}
\vspace{-8pt}
\end{center}

\begin{center}
salah.ali@fsc.bu.edu.eg, salahelgendi@yahoo.com
\end{center}

\vspace{0.3cm}

 \begin{abstract}
  In this paper, for Finsler surfaces,  we prove that  the T-condition and $\sigma T$-condition coincide.  For higher dimensions $n\geq 3$, we illustrate by  an example  that the  T-condition and $\sigma T$-condition    are not equivalent.    We show  that the non-homothetic  conformal change of a  Berwald (resp.  a Landsberg)  surface is  Berwaldian (resp. Landsbergian) if and only if   the $\sigma T$-condition is satisfied.  By solving the Landsberg's PDE, we classify all Finsler surfaces satisfying the T-condition, or equivalently the $\sigma T$-condition. Some examples are provided and studied.
 \end{abstract}

\noindent{\bf Keywords:\/}\, T-tensor; T-condition; $\sigma T$-condition;  Landsberg's PDE; Landsberg surfaces.

\medskip\noindent{\bf MSC 2020:\/}  53C60, 53B40.

\section{Introduction}
~\par

 The   $T$-tensor was introduced by M. Matsumoto \cite{ttensor}, it plays an important role in Finsler geometry and its  applications, especially, in  general relativity.   M. Hashiguchi \cite{Hashiguchi} studied the conformal change of Finsler metrics and showed that a Landsberg space remains a Landsberg space under any conformal change   if and only if its $T$-tensor vanishes.   Z. I. Szab\'o \cite{Szabo_$T$-tensor} proved that  a positive definite Finsler manifold with vanishing $T$-tensor is Riemannian. For more applications and details,  we refer,  for example, to \cite{Asanov2,Asanov1,Asanov_$T$-tensor}.

 \bigskip

In \cite{Asanov_$T$-tensor}, Asanov has studied the Finsler metrics with vanishing T-tensor, or in other words, the Finsler metrics satisfying the T-condition. So a Finsler metric satisfies the T-condition if the T-tensor vanishes. Moreover, in \cite{Elgendi-LBp} tackling the Landsberg's unicorn problem,  a weaker condition appeared.   In addition, later in \cite{Elgendi-ST_condition}, this condition is studied with more attentions and it is  called the $\sigma T$-condition.
A Finsler space $(M,F)$ is said to satisfy the   $\sigma T$-condition if $M$ admits a non-constant function $\sigma(x)$ such that
 $$\sigma_r T^r_{jk\ell}=0, \quad \sigma_r:=\frac{\partial \sigma}{\partial x^r}.$$
 Let $(M,F)$ be a Finsler space and $F$ be a positive definite metric. If the $\sigma T$-condition holds for \textit{every} $\sigma\in C^\infty(M)$,     then the T-tensor vanishes, i.e., the T-condition is satisfied. Therefore, by Szab\'o's observation $(M,F)$ is Riemannian.   So it will be more beneficial or interesting to consider the case when $\sigma T$-condition is satisfied for \textit{some} $\sigma\in C^\infty(M)$.

 \bigskip

  In \cite{Elgendi-LBp,Elgendi-ST_condition},  the $(\alpha,\beta)$-metrics that satisfy  the condition $\sigma_r T^r_{jkh}=0$  are characterized. An $(\alpha,\beta)$-metric $F$ is a metric on  the form $F=\alpha\phi(s)$, $s:=\frac{\beta}{\alpha}$.
It was shown that an  $(\alpha,\beta)$-metric with $n\geq 3$ satisfies the T-condition if and only if it is Riemannian or $\phi(s)$ has the following form
 \begin{equation}
 	\label{Eq:class_T_condition}
 	\phi(s)=    f(x)     s^{\frac{c b^2-1}{cb^2}}( b^2-s^2)^{\frac{1}{2cb^2}}
 \end{equation}
where $c$ is a constant and $f(x)$ is an arbitrary function on $M$ and $b^2:=\|\beta\|_\alpha$.  Also, an $(\alpha,\beta)$-metric  with  $n\geq 3$   satisfies the $\sigma T$-condition  if and only if the $T$-tensor vanishes  or $\phi(s)$ is given by
 \begin{equation}
	\label{Eq:class_sigmaT_condition}
\phi(s)=c_3 \,\exp\left(\int_0^s \frac{c_1\sqrt{b^2-t^2}+c_2 t}{t(c_1\sqrt{b^2-t^2}+c_2 t)+1}dt\right)
\end{equation}
where $c_1$, $c_2$ and $c_3$ are arbitrary constants.

It is worthy  to mention that the class \eqref{Eq:class_sigmaT_condition} has  been  already obtained by Z. Shen \cite{Shen_example}, in a completely different context,  with some restrictions on $\alpha$ and $\beta$.  Moreover, in \cite{Elgendi-LBp}, it was  shown  that the long existing problem of  Landsberg  non-Berwaldian spaces  is
related to the $\sigma T$-condition.

 \bigskip

  In the present paper, for higher dimensions, we show that T-condition and $\sigma T$-condition on Finsler manifolds are not equivalent. The classes  \eqref{Eq:class_T_condition} and \eqref{Eq:class_sigmaT_condition} are good illustration to this fact.   For concrete examples, see Examples \ref{Example-1} and \ref{Example-2}. We prove that the  T-condition and $\sigma T$-condition on Finsler surfaces  coincide.    As a result, we show  that a non-homothetic conformal change of a    Landsberg  surface is  Landsbergian if and only if   the T-tensor vanishes.  Moreover, we prove that a non-homothetic conformal change of a Finsler surface preserves the property of being Berwaldian if and only if the T-tensor vanishes or equivalently the $\sigma T$-condition is satisfied.

By solving the Landsberg's PDE, we characterize all Finsler surfaces with vanishing T-tensor, that is, a Finsler surface $(M,F)$ has vanishing T-tensor if and only if 
  \begin{equation*}
   F(x,y)=\sqrt{c_3(y^2)^2+(c_2c_3-4c_1+1)y^1y^2+c_2 (y^1)^2} \ e^{ \frac{(-c_2c_3+4c_1+1 )\operatorname{ arctanh}\left( \frac{2c_3 y^2+(c_2c_3-4c_1+1)y^1}{y^1\sqrt{c_2^2c_3^2-8c_1c_2c_3+16c_1^2-2c_2c_3-8c_1+1}} \right) }{\sqrt{c_2^2c_3^2-8c_1c_2c_3+16c_1^2-2c_2c_3-8c_1+1}}}
 \end{equation*}
  or
  \begin{equation*}
  F(x,y)=\sqrt{a(y^2)^2+by^1y^2+ (y^1)^2} \ e^{ - \frac{b}{\sqrt{b^2-4a}}\ \text{arctanh}\left( \frac{2a y^2+b y^1 }{ y^1\sqrt{b^2-4a}}\right) }
 \end{equation*}
  where $a$,  $b$, $c_1$, $c_2$ and $c_3$    are functions of $x^1$ and $x^2$.

\section{Preliminaries}

~\par

Let $M$ be an $n$-dimensional manifold,  $(TM,\pi_M,M)$ be the tangent bundle and $(\T M,\pi,M)$ be the subbundle of nonzero tangent vectors. The notation $C^\infty(M)$ stands for the $\Real$-algebra of smooth real-valued functions on $M$;  $\mathfrak{X}(M)$ stands for the $C^\infty(M)$-module of vector fields on $M$.  We denote by $(x^i) $ the local coordinates on the  manifold $M$, and  by $(x^i, y^i)$ the induced coordinates on the tangent bundle $TM$.  The vector $1$-form $J$ on $TM$ defined by $J = \frac{\partial}{\partial y^i} \otimes dx^i$ is  the natural almost-tangent structure of $T M$. The vertical vector field $\C=y^i\frac{\partial}{\partial y^i}$ on $TM$ is  the canonical or the
Liouville vector field.

A vector field $S\in \mathfrak{X}(\T M)$ is  a spray if $JS = \C$ and $[\C, S] = S$. Locally, a spray $S$ is given by
\begin{equation}
  \label{eq:spray}
  S = y^i \frac{\partial}{\partial x^i} - 2G^i\frac{\partial}{\partial y^i},
\end{equation}
where $G^i=G^i(x,y)$ are the {spray coefficients}.
A nonlinear connection is  an $n$-dimensional distribution (called the horizontal distribution) $H : u \in \tm \rightarrow H_u\subset T_u(\tm)$ and supplementary to the vertical distribution, that is,  for all $u \in \tm$, we have
\begin{equation}
  \label{eq:direct_sum}
 T_u(\tm) = H_u(\tm) \oplus V_u(\tm).
\end{equation}

Every spray S induces a canonical nonlinear connection through the corresponding horizontal and vertical projectors,
\begin{equation}
  \label{projectors}
    h=\frac{1}{2}  (Id + [J,S]), \,\,\,\,            v=\frac{1}{2}(Id - [J,S]).
\end{equation}
 With respect to the induced nonlinear connection, a spray $S$ is horizontal, which means that $S = hS$. Locally, the two projectors $h$ and $v$ can be expressed as follows
$$h= \delta_i\otimes dx^i, \quad\quad v=\paa_i\otimes \delta y^i,$$
where we use the notations
$$\delta_i:=\frac{\partial}{\partial x^i}-G^j_i(x,y)\paa_j, \quad \paa_i:=\frac{\partial}{\partial y^i},\quad \delta y^i=dy^i+G^i_j(x,y)dx^j, \quad G^j_i(x,y)=\paa_i G^j .$$
Moreover, the coefficients of the Berwald connection are given by
$$ G^h_{ij}={\paa_i G^h_j}.$$

\begin{definition}
An $n$-dimensional Finsler manifold  is a pair $(M,F)$, where $M$ is an $n$-dimensional  differentiable manifold and $F$ is a map  $$F: TM \To \Real ,\vspace{-0.1cm}$$  such that{\em:}
 \begin{description}
    \item[(a)] $F$ is smooth and strictly positive on $\T M$ and $F(x,y)=0$ if and only if $y=0$,
    \item[(b)]$F$ is positively homogeneous of degree $1$ in the directional argument $y${\em:}
    $\mathcal{L}_{\mathcal{C}} F=F$,
    \item[(c)] The metric tensor $g_{ij}= \paa_i\paa_j E  $ has rank $n$ on $\T M$, where $E:=\frac{1}{2}F^2$ is the energy function.
 \end{description}
 \end{definition}
In this case $(M, F)$ is called regular Finsler manifold. If $F$ satisfies   the  conditions (a)-(c) on a conic subset of $TM$, then       $(M, F)$   is called a conic   Finsler manifold.

\medskip

The \emph{Berwald tensor} (curvature) $G$    and the Landsbeg tensor $L$ are given, respectively,  by
\begin{equation}
\label{Berwald_curv.}
G=G^h_{ijk} dx^i\otimes dx^j\otimes dx^k\otimes\paa_h
\end{equation}
\begin{equation}
\label{Landsberg_Tensor}
L=L_{ijk} dx^i\otimes dx^j\otimes dx^k,
\end{equation} where $L_{ijk}=-\frac{1}{2}F G^h_{ijk}\paa_hF$, $G^h_{ijk}=\paa_kG^h_{ij}$ ,  see \cite{Shen-book}.
 \begin{definition}
A Finsler manifold $(M,F)$ is said to be \textit{Berwald} if the Berwald tensor $G^{h}_{ijk}$ vanishes identically, and   $(M,F)$  is called \textit{Landsberg} if the Landsberg tensor $L_{jkh}$ vanishes identically.
\end{definition}
 
The T-tensor plays an important role in Finsler geometry, it is introduced  by Matsumoto \cite{ttensor}.  For a Finsler   manifold $(M,F)$, the T-tensor  is defined by
\begin{equation} \label{T-tensor}
T_{hijk}=FC_{hijk}-F(C_{rij}C^{r}_{hk}+C_{rjh}C^{r}_{ik}+C_{rih}C^{r}_{jk})
+C_{hij}\ell_k+C_{hik}\ell_j +C_{hjk}\ell_i+C_{ijk}\ell_h,
\end{equation}
where $C_{ijk}:=\frac{1}{2}\dot{\partial}_kg_{ij}$ are  the components of the  Cartan tensor, $\ell_i:=\dot{\partial}_iF$,  $C_{ijkh}=\dot{\partial}_hC_{ijk},$  $C^{h}_{ij}=C_{\ell ij}g^{\ell h}$  and $g^{ij}$ are  the components of the inverse  metric tensor.

\section{T-condition and $\sigma T$-condition}

~\par

A Finsler space $(M,F)$ satisfies the \textit{T-condition} if  its $T$-tensor vanishes. In \cite{Asanov_$T$-tensor}, has studied the Finsler spaces satisfying the T-condition. Similarly, in \cite{Elgendi-ST_condition} the notion of  \textit{$\sigma T$-condition} is introduced. A Finsler space $(M,F)$ satisfies   the \textit{$\sigma T$-condition} if it admits  a non-constant function  $\sigma (x)$ such that $\sigma_hT^h_{ijk}=0$, $\sigma_h:=\frac{\partial \sigma}{\partial x^h}$.

Making use of the classes \eqref{Eq:class_T_condition} and \eqref{Eq:class_sigmaT_condition}, we give the following two examples.
The first example provides a Finsler metric  satisfying the T-condition and the second one satisfying the $\sigma T$-condition.

By making use of \cite{Elgendi-ST_condition}, we have the following Finsler metric that satisfies the T-condition.
\begin{example}\label{Example-1}
Let $M=\mathbb{R}^n$, $n\geq 3$   and $\alpha=|y|$ be the Euclidean norm. Assuming that  $\beta=y^1$, then $b^2=1$. Let $F$ be the $(\alpha,\beta)$-metric given by
$$F=\alpha \phi(s), \quad \phi(s)=\sqrt{s}(1-s^2)^{1/4}.$$
The Finsler manifold $(M,F)$ satisfies the T-condition, that is, $T^h_{hijk}=0$.
\end{example}

By using \cite{Elgendi-solutions}, we have the following example.
\begin{example}\label{Example-2}
	Let $M=\mathbb{R}^n$, $n\geq 3$  and  $\beta=f(x^1)y^1$ and $\alpha=f(x^1)\sqrt{(y^1)^2+\varphi(\hat{y})}$, where $f(x^1)$ is a positive smooth function on $\Real$ and $\varphi$ is an arbitrary quadratic function in $\hat{y}$ and $\hat{y}$ stands for the variables $y^2,...,y^n$.
	Let the Finsler function $F$ on $\mathbb{R}^n$ be an a special $(\alpha,\beta)$-metric given by
$${F}=\left(a \beta+\sqrt{\alpha^2-\beta^2}\right)\,\exp{\left(\frac{a \beta}{a\beta+\sqrt{\alpha^2-\beta^2}}\right)}, \quad a\neq 0.$$
One can use the package \cite{NF_Package} to find that $T^1_{ijk}=0$ and some other components $T^\mu_{ijk}$ are non-zero.
Assuming that   $\sigma_h=b_h$, taking into account the fact that $b_1=f(x^1)\neq 0$, $b_2=...=b_n=0$,  we conclude that
$$\sigma_hT^h_{ijk}=b_hT^h_{ijk}=0.$$
That is  $(M,F)$ satisfies the  $\sigma T$-condition.
\end{example}

 The above two examples show that the T-condition and $\sigma T$-condition on higher dimensional manifolds are not equivalent. 
 
 \subsection{Finsler surfaces}

For the two-dimensional case,  in \cite{Berwald}, Berwald has introduced a  frame for the positive definite surfaces.  Later, in \cite{Matsumoto},  B\'{a}sc\'{o} and Matsumoto    have modified the Berwald frame to cover   the non positive definite surfaces. The  modified frame of a Finsler surface $(M,F)$  is given by $(\ell^i,m^i)$, where  $m^i$ is a vector which is orthogonal to the supporting element $\ell_i$ and the co-frame is $(\ell_i,m_i)$. Moreover, we have
 $$m_i=g_{ij}m^j, \quad m^im_i=\varepsilon, \quad \ell^im_i=0,$$
 where   $g_{ij}=\ell_i\ell_j+\varepsilon m_im_j$, $\varepsilon=\pm 1$ and the sign $\varepsilon$ is  called the signature of $F$. In the positive definite case, $\varepsilon=+1$.

 \medskip
 For a scalar function $L$ on $\T M$, we  write the horizontal covariant derivative of $L$ with respect to Berwald connection as follows:
 $$L_{|i}=L_{,1} \ell_i+L_{,2}m_i,$$
 where
 $L_{,1}=\ell^iL_{|i}$, $L_{,2}=m^iL_{|i}.$
 Also, we can write
 $$F\paa_i L= L_{;1}\ell_i+L_{;2}m_i.$$
 \begin{property}
 	\label{property}
If $L$ is homogeneous of degree $0$ in $y$, then  $L_{;1}=0$ and hence
$$F\paa_i L= L_{;2}m_i.$$
 \end{property}

 \begin{lemma}\label{Lemma:Matsumoto}\cite{Matsumoto}
 For Finsler surface $(M,F)$, we have  the following associated geometric objects:
 \begin{description}
\item[(a)] The Cartan tensor: $C_{ijk}= \frac{I}{F}m_im_jm_k$,

\item[(b)] The Berwald tensor: $G^h_{ijk}= \frac{1}{F}\{-2I_{,1}\ell^h+(I_{,2}+I_{,1;2})m^h\}m_im_jm_k$,

\item[(c)] The Landsberg tensor: $L_{ijk}=-\frac{1}{2}F\ell_hG^h_{ijk}= I_{,1}m_im_jm_k$,

\item[(d)] The T-tensor: $T^h_{ijk}=\frac{1}{F}I_{;2}m^hm_im_jm_k$.
\end{description}
where $I$ is a $0$-homogeneous function in $y$ and called the main scalar of the manifold $(M,F)$.
 \end{lemma}
\begin{definition}\cite{Matsumoto}
A two dimensional space $(M,F)$ is Landsbergian if
$$I_{,1}=0.$$
Also,  $(M,F)$ is Berwaldian if
$$I_{,1}=I_{,2}=0.$$
\end{definition}

\begin{property}\cite{Hashiguchi} \label{T-tensor=0}
A Finsler surface has vanishing T-tensor if and only if $I$ is a point function that is $I=I(x)$ which is equivalent to $I_{;2}=0$.
\end{property}
 For Finsler surfaces, we have the following   theorem.

\begin{theorem}\label{main_theorem}
A Finsler surface $(M,F)$ satisfies the T-condition if and only if $(M,F)$ satisfies  $\sigma T$-condition.
\end{theorem}

\begin{proof}
Let $(M,F)$ be a Finsler surface with vanishing T-tensor, that is, the T-condition is satisfied. Then, it is obvious that the $\sigma T$-condition is satisfied.

Conversely, let $(M,F)$ satisfy the $\sigma T$-condition. Then, there is a function $\sigma(x)$ on $M$ such that
$$\sigma_rT^r_{jkh}=0, \quad \sigma_r:=\frac{\pa \sigma}{\pa x^r}.$$
Therefore, by Lemma \ref{Lemma:Matsumoto} (d),  we have
$$ \sigma_rT^r_{jkh }=\frac{I_{;2}}{F} \sigma_r m^rm_jm_km_h=0.$$
Since $m_j\neq 0$, then we must have $I_{;2}=0$ or $\sigma_r m^r=0$. If $I_{;2}=0$, then the T-tensor vanishes and we are done. Now, if $\sigma_r m^r=0$, then we have
$$\sigma_1m^1+\sigma_2m^2=0,  \quad  m^1  \paa_1E+m^2 \paa_2E=0,$$
where $E=\frac{F^2}{2}$.
The above two equations can be seen as algebraic equations at every point of $\T M$.  Since $m^1$ and $m^2$ are non-zero at each point of $\T M$, then we must have
$$\sigma_1\ \paa_2E-\sigma_2\ \paa_1E=0.$$
Differentiating the above equation with respect to $y^1$ and $y^2$ respectively, we have
\begin{eqnarray*}
 \sigma_1\ g_{12}-\sigma_2\ g_{11}&=&0,\\
  \sigma_1\ g_{22}-\sigma_2\ g_{12}&=&0.
\end{eqnarray*}
Since $\det(g_{ij})=g_{12}^2-g_{11} g_{22}\neq 0$, then we must have $\sigma_1=\sigma_2=0$ at each point of $\T M$. This implies that $\sigma$ is constant which is a contradiction.
 Hence, since $\sigma(x)$ is not constant, then $I_{;2}=0$ and this means that the T-tensor vanishes. This completes the proof.
\end{proof}

\subsection{Conformal Change}

 Now, we consider the conformal change of  a  Finsler metric $F$, namely,

  \begin{equation}\label{conformal_F}
  \overline{F}=e^{\sigma(x)}F,
  \end{equation}
  where $\sigma(x)$ is a smooth function on $M$.

   It should be noted that all geometric objects associated with the transformed space $(M,\overline{F})$ will be elaborated by barred symbols.

\begin{lemma}\cite{Elgendi-LBp}
Under the conformal change \eqref{conformal_F}, the Berwald tensor  transforms as follows
$$\overline{G}^i_{jkh}=G^i_{jkh}+B^i_{jkh},$$
where
  \begin{align}\label{B-T-tensor}
\nonumber   B^i_{jkh}=& F\sigma_r\dot{\partial}_hT^{ri}_{jk}+\sigma_r(T^{ri}_{jh} \ell_k+T^{ri}_{kh}\ell_j+T^{ri}_{jk}\ell_h-T^r_{jkh}\ell^i-T^i_{jkh}\ell^r)\\
\nonumber   &-F\sigma_r(T^i_{sjh}C^{sr}_{k}+T^r_{skh}C^{si}_{j}+T^r_{sjh}C^{si}_{k}+T^i_{skh}C^{sr}_{j}-T^{ri}_{sh}C^s_{jk}-T^s_{jkh}C^{ri}_{s})\\
   &+\sigma_r(C^{ri}_{j}h_{kh}+C^{ri}_{k}h_{jh}+2 C^{ir}_hh_{jk}-C^r_{jk}h^i_h-C^i_{jk}h^r_h-2 C_{jkh}h^{ir})    \\
 \nonumber  &+F^2\sigma_r[C^{t}_{hj}S^{\,\, ir}_{t\quad k}+C^{t}_{hk}S^{\,\, ri}_{t\quad j}-C^{ti}_{h}S^{\quad \, r}_{tjk}-C^{tr}_{h}S^{\quad \, i }_{tkj}-C^{ti}_{j}S^{\quad \, r}_{thk}-C^{tr}_{k}S^{ \quad \, i}_{th j}
   ],
\end{align}
where $S^{\,\, h}_{i\,\, jk}=C^{r}_{ik}C^{h}_{rj}-C^{r}_{ij}C^{h}_{rk}$ is the v-curvature of Cartan connection.
\end{lemma}

\begin{lemma}\cite{Elgendi-LBp}
Under the conformal change \eqref{conformal_F}, the Landsberg tensor has the following transformation
\begin{equation}\label{confrmal_L_tensor}
   \overline{L}_{jkh}=e^{2\sigma} L_{jkh}+e^{2\sigma} F\sigma_rT^r_{jkh }.
   \end{equation}
\end{lemma}

\begin{remark}
 In  1976, Hashiguchi \cite{Hashiguchi}  showed   that a  Landsberg space remains  Landsberg by \textit{every} conformal change if and only if the T-tensor vanishes. However,   there are     Landsberg spaces (with non vanishing  T-tensor) which remain Landsberg under \textit{some}   conformal transformation, see \cite{Elgendi-LBp, Elgendi-solutions}. But there is no a  Landsberg surface $(M,F)$ with non-vanishing T-tensor which remains  Landsberg under a conformal transformation, as be shown in the following theorem.   
\end{remark}

\begin{theorem}\label{Th:3.10}
The non homothetic conformal transformation of a   Landsberg surface $(M,F)$ is Landsbergian if and only if the T-tensor of $(M,F)$ vanishes.
\end{theorem}
\begin{proof}
Let $(M,F)$ be a Landsberg surface, then $L_{ijk}=0$. Now by \eqref{confrmal_L_tensor}, we have
$$\overline{L}_{jkh}=  e^{2\sigma} F\sigma_rT^r_{jkh }.$$
Assume that $(M,\overline{F})$ is Landsbergian, then we have
$$\overline{L}_{jkh}=  e^{2\sigma} F\sigma_rT^r_{jkh }= 0.$$
That is, $\sigma_rT^r_{jkh }= 0.$
Using Theorem \ref{main_theorem}, we conclude that $T^h_{ijk}=0$.

Conversely, assume that $T^h_{ijk}=0$, then by \eqref{confrmal_L_tensor} we have
$$\overline{L}_{jkh}= e^{2\sigma} L_{jkh}.$$
Consequently, the result follows.
\end{proof}

  In \cite{Matsumoto}, B\'{a}cs\'{o} and Matsumoto   proved that a Landsberg surface that satisfies the T-condition is Berwaldain. Making use of Theorem \ref{main_theorem}, we have the following generalized version of B\'{a}cs\'{o} and Matsumoto's result.
  \begin{theorem}
  	A Landsberg surface satisfying the $\sigma T$-condition is Berwaldian.
  \end{theorem}
\begin{proof}
	Let $(M,F)$ be a Landsberg surface and satisfies the  $\sigma T$-condition. Then by Theorem \ref{main_theorem}, the $T$-condition is satisfied. Hence, by \cite[Theorem 2]{Matsumoto},  $(M,F)$ is Berwaldian.
\end{proof}

Now, let's  request the conformal transformation to preserve the property of being Berwaldian, so we have the  following theorem.
\begin{theorem}
The non-homothetic conformal transformation of a   Berwald surface $(M,F)$ is Berwaldian if and only if   $(M,F)$ satisfies the $\sigma T$-condition.
\end{theorem}
\begin{proof}
	By \cite{Matsumoto}, we have
	 $$F\paa_jm^i=-(\ell^i+\varepsilon I m^i)m_j, \quad F\paa_jm_i=-(\ell_i-\varepsilon I m)i)m_j.$$
Now, in terms of Berwald frame and making use of \eqref{B-T-tensor}, we get
\begin{align*}
\paa_h T^{ri}_{jk}&= \paa_h\left( \frac{I_{;2}}{F}m^rm^im_jm_k\right)\\
&=\frac{\paa_h I_{;2}}{F}m^rm^im_jm_k-\frac{I_{;2}}{F^2}m^rm^im_jm_k\ell_h-\frac{I_{;2}}{F^2}(\ell^r+\varepsilon I m^r)m_hm^im_jm_k\\
&-\frac{I_{;2}}{F^2}(\ell^i+\varepsilon I m^i)m_hm^rm_jm_k-\frac{I_{;2}}{F^2}m^rm^i(\ell_j-\varepsilon I m_j)m_hm_k-\frac{I_{;2}}{F^2}m^rm^i(\ell_k-\varepsilon I m_k)m_hm_j
\end{align*}
Then, we  under the conformal transformation \eqref{conformal_F} and keeping in mind that the components $S^h_{ijk}$ of the v-curvature of any surfaces vanish, the Berwald tensor  transforms as follows
$$\overline{G}^i_{jkh}=G^i_{jkh}+B^i_{jkh},$$
where
 \begin{align*}
   B^i_{jkh}&= ( \paa_h I_{;2})\sigma_rm^rm^im_jm_k-\frac{I_{;2}}{F}\sigma_rm^rm^im_jm_k\ell_h-\frac{I_{;2}}{F }\sigma_r(\ell^r+\varepsilon I m^r)m_hm^im_jm_k\\
&-\frac{I_{;2}}{F }\sigma_r(\ell^i+\varepsilon I m^i)m_hm^rm_jm_k-\frac{I_{;2}}{F }\sigma_rm^rm^i(\ell_j-\varepsilon I m_j)m_hm_k\\
   &-\frac{I_{;2}}{F }\sigma_rm^rm^i(\ell_k-\varepsilon I m_k)m_hm_j+\frac{I_{;2}}{F }\sigma_r(m^rm^im_jm_h \ell_k+m^rm^im_km_h\ell_j\\
   &+m^rm^im_jm_k\ell_h-m^rm_km_jm_h\ell^i- m^im_jm_hm_k\ell^r)-\frac{2\varepsilon I I_{;2}}{F }\sigma_rm^r m^im_jm_h m_k\\
   &= (\paa_h I_{;2})\sigma_rm^rm^im_jm_k- \frac{2I_{;2}}{F }\sigma_r\ell^rm_hm^im_jm_k
-\frac{2I_{;2}}{F }\sigma_r\ell^im_hm^rm_jm_k
\\&-\frac{2\varepsilon I I_{;2}}{F }\sigma_rm^r m^im_jm_h m_k.\\
  \end{align*}
  Since $I_{;2}$ is homogeneous of degree $0$, then by Property  \ref{property}, we  have
  $$F\paa_h I_{;2}=I_{;2;1}\ell_h+I_{;2;2}m_h=I_{;2;2}m_h$$ then  $B^i_{jkh}$ can be written as follows
  \begin{align*}
   B^i_{jkh}&=  \left( I_{;2;2}\sigma_rm^r- \frac{2I_{;2}}{F }\sigma_r\ell^r
-\frac{2\varepsilon I I_{;2}}{F }\sigma_rm^r\right)m^im_jm_h m_k-\frac{2I_{;2}}{F }\sigma_r\ell^im_hm^rm_jm_k.
  \end{align*}
  Assuming that  $(M,F)$ and $(M,\overline{F})$ are both Berwaldian, then the difference tensor $ B^i_{jkh}$ vanishes identically. So, we have  $ B^i_{jkh}=0$ and since $m^i$ and $\ell^i$ are independent, then we must have $I_{;2}\sigma_r m^r$. Hence, $I_{;2}=0$ or $\sigma_r m^r=0$ and consequently, by Lemma \ref{Lemma:Matsumoto} (d), the $\sigma T$-condition is satisfied.
\end{proof}

\section{Finsler surfaces satisfying   the T-condition}

To find explicit formulae of the Finsler surfaces that satisfy the T-condition (with vanishing T-tensor), we recall       the following new look of Finsler surfaces \cite{New_surfaces}. 
\begin{lemma}[\cite{New_surfaces}]\label{Lemma_G1_G2}
   Let $F$ be a Finsler function on a two-dimensional manifold $M$, then $F$ can be  written  in the form
\begin{equation}\label{Finsler_Function}
  F=\left\{
\begin{array}{ll}
   \abs{y^1}f(x,\varepsilon u),      & \quad u=\frac{y^2}{y^1}, \   y^1\neq 0, \ \varepsilon:=\operatorname{sgn}(y^1) \\
      0 , & \quad  y^1=y^2=0\\
      \abs{y^2} f(x,{\epsilon} v),   &\quad v=\frac{y^1}{y^2}, \    y^2\neq 0,\ {\epsilon}:=\operatorname{sgn}(y^2) \\
\end{array} 
\right.
\end{equation}
where $f(x,\varepsilon u):=F(x,\varepsilon ,\varepsilon u)$ is a positive smooth function on $M\times\Real$   and $|\cdot |$ is the absolute value.

Moreover, for the expression $ F=\abs{y^1}f(x,\varepsilon u)$  the  coefficients $G^1$ and $G^2$ of the geodesic spray are  given by
\begin{equation}\label{G^1_G^2}
  G^1=f_1(x,u)(y^1)^2, \quad  G^2=f_2(x,u)(y^1)^2,
\end{equation}
where  the functions $f_1$ and $f_2$ are smooth functions on $M\times\Real$ and  given as follows
\begin{equation}\label{G_f1}
f_1=\frac{(\pa_1f+u\pa_2f)f''-(\pa_1f'+u\pa_2f'-\pa_2f)f'}{2ff''},
\end{equation}
\begin{equation}\label{G_f2}
f_2=\frac{u(\pa_1f+u\pa_2f)f''+(\pa_1f'+u\pa_2f'-\pa_2f)( f- uf')}{2ff''},
\end{equation}
where    $f'$ (resp. $f''$) is the first (resp. the second) derivative of  $f$  with respect to $ u$ and so on.
\end{lemma}

\begin{remark}
 It should be noted that if we start by regular Finsler function $F$, then the Finsler function $F(x,y)=\abs{y^1} f(x,\varepsilon u)$ is regular although the function $u$ has a singularity at $y^1=0$.
As an  example (cf. \cite[Example 1.2.2 Page 15]{shen-book1}):
$$F(x,y)=\sqrt{(y^1)^2+(y^2)^2}+B y^1=|y^1|\left(\sqrt{1+u^2}+\varepsilon B\right).$$
In this example $f(x, \varepsilon u)=\sqrt{1+u^2}+\varepsilon B$.
Since $F=0 $ only on the zero section, then away from the zero section at each $x\in M$, at least one   of the $y$'s is non zero, so without loss of generality, we assume that $y^1\neq 0$.
\end{remark}

\begin{lemma}[\cite{New_surfaces}]
  The components $L_{ijk}$ of the Landsberg curvature are given by
  \begin{equation}\label{Lands_tensors}
  \begin{split}
       L_{111}&=\frac{u^3f}{2}(f_1'''\ell_1+f_2'''\ell_2), \quad  L_{112}=-\frac{u^2f}{2}(f_1'''\ell_1+f_2'''\ell_2),  \\
        L_{122}&=-\frac{uf}{2}(f_1'''\ell_1+f_2'''\ell_2), \quad  L_{222}=-\frac{f}{2}(f_1'''\ell_1+f_2'''\ell_2).
\end{split}
  \end{equation}
\end{lemma}

\begin{lemma}[\cite{New_surfaces}]\label{Landsberg_PDE}
  Any  two dimensional Finsler  manifold $(M,F)$ in the form \eqref{Finsler_Function} is  Landsbergian if and only if the following PDE
  \begin{equation}
  \label{Eq:Landsberg_PDE}
  f_1'''\ell_1+f_2'''\ell_2=0
  \end{equation}
is satisfied. The above PDE is called the  Landsberg's PDE.
\end{lemma}

Let's define  the function $Q$ as follows
$$Q:=\frac{f'}{f-u f'}.$$
  Moreover,  the function $f$ is given by
\begin{equation}\label{Q_f}
f(x,u)=\exp\left(\int \frac{Q}{1+uQ}du\right).
\end{equation}
\begin{property}\label{property} For any Finsler surface the function $Q$ has the property
$$Q'\neq0.$$
\end{property}
\begin{proof}
  Assume that $Q'=0$. This implies $Q=\theta(x)$ and hence we have
$$\frac{Q}{1+uQ}=\frac{\theta(x)}{1+u\theta(x)}.$$
Therefore, by using \eqref{Finsler_Function} and \eqref{Q_f}, we have
$$F=\abs{y^1}\exp(\ln (1+u\theta(x))= \varepsilon (y^1+\theta(x)y^2).$$
This means that the Finsler function is linear and hence the metric tensor is degenerate which is a contradiction.
\end{proof}
Consider  the conformal transformation
\begin{equation}\label{Conformal_change}
  \overline{F}=e^{\sigma(x)} F=\abs{y^1}e^{\sigma(x)} f(x,u).
\end{equation}

Keeping in mind the Property \ref{property}, we have the following.
\begin{proposition}
  Under the conformal transformation \eqref{Conformal_change}, we have
  \begin{equation}\label{L_tensor_conformal}
    \overline{f}_1'''+ \overline{Q}\overline{f}_2'''=f_1'''+Qf_2'''+\frac{2 \sigma_1 QQ'Q'''-3\sigma_1 QQ''^2-2 \sigma_2 Q'Q'''+3\sigma_2 Q''^2}{2 Q'^2}
  \end{equation}
\end{proposition}

\begin{proof}
  Consider  the conformal transformation \eqref{Conformal_change}, then we have

$$\overline{f}'=e^{\sigma} f', \quad \overline{f}''=e^{\sigma} f'',$$
$$\partial_1\overline{f}=e^{\sigma} \partial_1f+e^{\sigma} f \partial_1\sigma  , \quad \partial_2\overline{f}=e^{\sigma} \partial_2f+e^{\sigma} f\partial_2\sigma,  $$
$$\partial_1\overline{f}'=e^{\sigma} \partial_1f'+e^{\sigma} f' \partial_1\sigma  , \quad \partial_2\overline{f}'=e^{\sigma} \partial_2f'+e^{\sigma} f'\partial_2\sigma.$$

 By  making use of the above relations together with the help of the quantities $Q=\frac{  f'}{f-uf'}$,  $Q'=\frac{ff''}{(f-uf')^2}$, then  \eqref{G_f1} and \eqref{G_f2} lead to
\begin{eqnarray*}
   \overline{f}_1
    &=&  f_1+  \frac{\partial_1\sigma+u\partial_2\sigma}{2}+\frac{\partial_2\sigma}{2}\frac{Q}{Q'}- \frac{\partial_1\sigma}{2} \frac{Q^2}{Q'},
\end{eqnarray*}
\begin{eqnarray*}
   \overline{f}_2
    &=&  f_2+  \frac{u(\partial_1\sigma+u\partial_2\sigma)}{2}+\frac{\partial_1\sigma}{2}\frac{Q}{Q'}-\frac{\partial_2\sigma}{2}\frac{1}{Q'}.
\end{eqnarray*}

Moreover, we have  the formulae
$$\left(\frac{1}{Q'}\right)'''=-\frac{Q'^2Q''''-6Q'Q''Q'''+6Q''^3}{Q'^4},$$
$$\left(\frac{Q}{Q'}\right)'''=-\frac{2Q'^3Q'''-3Q'^2Q''^2+QQ'^2Q''''-6QQ'Q''Q'''+6QQ''^3}{Q'^4},$$
$$\left(\frac{Q^2}{Q'}\right)'''=-\frac{Q\left(4Q'^3Q'''-6Q'^2Q''^2+QQ'^2Q''''-6QQ'Q''Q'''+6QQ''^3\right)}{Q'^4}.$$
Now, since $\overline{Q}=Q$  and using the above formulae of $\overline{f}_1$ and $\overline{f}_2$, then straightforward calculations yield \eqref{L_tensor_conformal}.
\end{proof}

\begin{theorem}\label{Th:class-with-0-T}
  The Landsberg tensor of a Finsler surface $(M,F)$ is invariant under the conformal change \eqref{Conformal_change} if and only if
 \begin{equation}\label{Eq:Formula_1}
   f(x,u)=\sqrt{c_3u^2+(c_2c_3-4c_1+1)u+c_2} \ e^{ \frac{(-c_2c_3+4c_1+1 )\operatorname{ arctanh}\left( \frac{2c_3 u+c_2c_3-4c_1+1}{\sqrt{c_2^2c_3^2-8c_1c_2c_3+16c_1^2-2c_2c_3-8c_1+1}} \right) }{\sqrt{c_2^2c_3^2-8c_1c_2c_3+16c_1^2-2c_2c_3-8c_1+1}}}
 \end{equation}
  or
  \begin{equation}\label{Eq:Formula_2}
   f(x,u)=\sqrt{au^2+bu+ 1} \ e^{ - \frac{b}{\sqrt{b^2-4a}}\ \text{arctanh}\left( \frac{2a u+b }{ \sqrt{b^2-4a}}\right) }
 \end{equation}
  where $c_1$, $c_2$, $c_3$, $a$ and $b$ are functions of $x^1$ and $x^2$.
\end{theorem}
\begin{proof}
The components of the Landsberg tensor are given by \eqref{Lands_tensors}. The common term in all these components is
$$f_1'''\ell_1+f_2'''\ell_2=f_1'''+f_2'''\varepsilon f'=\varepsilon( f-uf')(f_1'''+Qf_2''')$$
where $\ell_1=\paa_1 F=\varepsilon( f-uf')$ and $\ell_2=\paa_2 F=\varepsilon f'$.
It is clear that all components of the Landsberg tensor are invariant under the conformal transformation \eqref{Conformal_change} if and only if the quantity $f_1'''+Qf_2'''$ is itself invariant.

Now, using making use of \eqref{L_tensor_conformal} the quantity $f_1'''+Qf_2'''$ is   invariant if and only if
$$\frac{2 \sigma_1 QQ'Q'''-3\sigma_1 QQ''^2-2 \sigma_2 Q'Q'''+3\sigma_2 Q''^2}{2 Q'^2}=0.$$
This implies
$$(\sigma_1 Q-\sigma_2)(2Q'Q'''-3Q''^2)=0.$$
By Property \ref{property}, the choice $\sigma_1 Q-\sigma_2=0$ implies a contradiction. Therefore, we have
$$2Q'Q'''-3Q''^2=0.$$

 If $Q''=0$, then $Q=au+b$.
  Now, we have
$$\frac{Q}{1+uQ}=\frac{au+b}{au^2+bu+1}=\frac{2au+b}{2(au^2+bu+1)}-\frac{2ab}{b^2-4a-(2au+b)^2}.$$
 Hence,
    $$\int \frac{Q}{1+uQ} du=\frac{1}{2}\ln{(a u^2+bu+1)}-\frac{ b  }{\sqrt{b^2-4a}} \operatorname{arctanh}\left( \frac{2a u+b }{ \sqrt{b^2-4a}}\right) .$$
  By substituting into \eqref{Q_f}, we have
    \begin{equation*}
   	f=\sqrt{au^2+bu+ 1} \ e^{ - \frac{b}{\sqrt{b^2-4a}}\ \text{arctanh}\left( \frac{2a u+b }{ \sqrt{b^2-4a}}\right) }.
   \end{equation*}
   Where $ a,b$ are functions of $x^1$ and $x^2$.

 Now assume that $Q''\neq 0$. Then the above PDE can be rewritten in the form
  $$1+2\left( \frac{Q'}{Q''}\right)'=0.$$
  Moreover, the above PDE has the solution
  $$\frac{Q'}{Q''}=-\frac{1}{2}u+c_1.$$
 Furthermore, we can find $Q'$, since
  $$\frac{Q''}{Q'}=\frac{2}{2c_1-u}.$$
  Which gives easily the formula of $Q'$ as follows
   \begin{equation*}\label{Q_t}
   Q'=\frac{c_2}{(2c_1-u)^2}.
   \end{equation*}
  That is, we get

   \begin{equation*}\label{Q}
   Q=\frac{c_2}{2c_1-u}+c_3,
  \end{equation*}
  where $c_1,c_2,c_3$ are arbitrary functions on    $M$.
  Now, we have
 $$\frac{Q}{1+uQ}=\frac{-c_3u+2c_1c_3+c_2}{-c_3u^2+(2c_1c_3+c_2-1)u+2c_1}$$
 which can be rewritten in the following useful form
  $$\frac{Q}{1+uQ}=\frac{1}{2}\frac{2c_3u-c+2}{c_3u^2-(c-2)u-2c_1}+\frac{2c c_3}{(c^2-4c_2)-(2c_3u-c+2)^2}.$$
  Hence, we have
    $$\int \frac{Q}{1+uQ} du=\frac{1}{2}\ln{(c_3u^2-(c-2)u-2c_1)}+\frac{ c  }{\sqrt{c^2-4c_2}} \operatorname{arctanh}\left( \frac{2c_3 u-(c-2) }{ \sqrt{c^2-4c_2}}\right) .$$
  \end{proof}
By making use of \eqref{confrmal_L_tensor},  Theorem \ref{main_theorem}, \eqref{Lands_tensors} and Theorem \ref{Th:class-with-0-T}, we can prove the following theorem.
\begin{theorem}
  A Finsler surface $(M,F)$ has vanishing T-tensor if and only if the function $f(x,u)$ is given by \eqref{Eq:Formula_1} or \eqref{Eq:Formula_2}.
\end{theorem}

 It should be noted that the two   classes \eqref{Eq:Formula_1} and  \eqref{Eq:Formula_2} are not Landsbergian in general.  Since all Landsberg surfaces with vanishing T-tensor are Berwaldian cf. \cite{Matsumoto}, then we have the following corollary.
 \begin{corollary}
   If the  classes \eqref{Eq:Formula_1} and  \eqref{Eq:Formula_2} are Landsbergian then they must be Berwaldian.
\end{corollary}

\begin{remark}
  In terms of $y^1$ and $y^2$, the classes \eqref{Eq:Formula_1} and \eqref{Eq:Formula_2}  are given as follows
  \begin{equation*}
   F(x,y)=\sqrt{c_3(y^2)^2+(c_2c_3-4c_1+1)y^1y^2+c_2 (y^1)^2} \ e^{ \frac{(-c_2c_3+4c_1+1 )\operatorname{ arctanh}\left( \frac{2c_3 y^2+(c_2c_3-4c_1+1)y^1}{y^1\sqrt{c_2^2c_3^2-8c_1c_2c_3+16c_1^2-2c_2c_3-8c_1+1}} \right) }{\sqrt{c_2^2c_3^2-8c_1c_2c_3+16c_1^2-2c_2c_3-8c_1+1}}}
 \end{equation*}
  or
  \begin{equation*}
  F(x,y)=\sqrt{a(y^2)^2+by^1y^2+ (y^1)^2} \ e^{ - \frac{b}{\sqrt{b^2-4a}}\ \text{arctanh}\left( \frac{2a y^2+b y^1 }{ y^1\sqrt{b^2-4a}}\right) }
 \end{equation*}
  where $c_1$, $c_2$, $c_3$, $a$ and $b$ are functions of $x^1$ and $x^2$.
\end{remark}

\section*{Declarations}

 \medskip

  \noindent \textbf{Ethical Approval:}  Not applicable.

  \medskip

\noindent \textbf{Competing interests:}  The author  declares no conflict of interest.

 \medskip

\noindent \textbf{Authors' contributions:} The author wrote the whole manuscript.

  \medskip

\noindent \textbf{Funding:} Not applicable.

 \medskip

\noindent \textbf{Availability of data and materials:}   Not applicable.

\providecommand{\bysame}{\leavevmode\hbox
to3em{\hrulefill}\thinspace}
\providecommand{\MR}{\relax\ifhmode\unskip\space\fi MR }
\providecommand{\MRhref}[2]{%
  \href{http://www.ams.org/mathscinet-getitem?mr=#1}{#2}
} \providecommand{\href}[2]{#2}

\end{document}